\documentclass{amsart}
\usepackage{}
\usepackage{amsmath,amsthm,amssymb,WU}
\usepackage{amsfonts}
\usepackage{amsthm}
\usepackage[leqno]{amsmath}
\usepackage{latexsym,amsfonts,amssymb}
\usepackage{hyperref}
\usepackage{enumerate}
\usepackage[all]{xy}
\usepackage{amscd}

\pdfoutput1

\theoremstyle{bfupright head,slanted body}
\newtheorem{res}{}[section]             \newtheorem*{res*}{}

\theoremstyle{bfupright head,upright body}
               \newtheorem*{bfhpg*}{}
\theoremstyle{numbered paragraph}

\theoremstyle{fixed bf head,slanted body}
\newtheorem{thm}[res]{Theorem}          \newtheorem*{thm*}{Theorem}
      \newtheorem*{prp*}{Proposition}
        \newtheorem*{cor*}{Corollary}
            \newtheorem*{lem*}{Lemma}

\theoremstyle{fixed bf head,upright body}
       \newtheorem*{dfn*}{Definition}
      \newtheorem*{obs*}{Observation}
           \newtheorem*{rmk*}{Remark}
          \newtheorem*{exa*}{Example}

\newcommand{\pgref}[1]{(\ref{#1})}
\newcommand{\thmref}[2][Theorem~]{#1\pgref{thm:#2}}

\renewcommand{\eqref}[1]{\pgref{eq:#1}}

\newcommand{\is}{\cong}

\newcommand{\Ker}[1]{\nobreak{\operatorname{Ker}#1}}
\newcommand{\Coker}[1]{\nobreak{\operatorname{Coker}#1}}

\newcommand{\fd}[2][R]{\operatorname{fd}_{#1}#2}
\newcommand{\fdd}[2][\widehat{R}]{\operatorname{fd}_{#1}#2}

\newcommand{\dif}[2][]{{\partial}^{#2}_{#1}}
\newcommand{\Co}[2][]{\operatorname{C}_{#1}(#2)}
\newcommand{\Cy}[2][]{\operatorname{Z}_{#1}(#2)}
\renewcommand{\H}[2][]{\operatorname{H}_{#1}(#2)}

\newcommand{\id}[2][R]{\operatorname{id}_{#1}#2}
\newcommand{\idd}[2][\widehat{R}]{\operatorname{id}_{#1}#2}
\newcommand{\pd}[2][R]{\operatorname{pd}_{#1}#2}

\newcommand{\Hom}[3][R]{\operatorname{Hom}_{#1}(#2,#3)}

\newcommand{\RHom}[3][\widehat{R}]{\operatorname{\mathbf{R}Hom}_{#1}(#2,#3)}
\newcommand{\RHomm}[3][R]{\operatorname{\mathbf{R}Hom}_{#1}(#2,#3)}
\newcommand{\Ext}[4][R]{\operatorname{Ext}_{#1}^{#2}(#3,#4)}

\newcommand{\tp}[3][R]{\nobreak{#2\otimes_{#1}#3}}

\newcommand{\Ltp}[3][\widehat{R}]{\nobreak{#2\otimes_{#1}^{\mathbf{L}}#3}}

\newcommand{\Ltpp}[3][R]{\nobreak{#2\otimes_{#1}^{\mathbf{L}}#3}}

\newcommand{\Tor}[4][R]{\operatorname{Tor}^{#1}_{#2}(#3,#4)}
\newcommand{\Gidd}[2][R]{\operatorname{Gid}_{#1}#2}
\newcommand{\Gid}[2][\widehat{R}]{\operatorname{Gid}_{#1}#2}
\newcommand{\lra}{\longrightarrow}

\newcommand{\xra}[2][]{\xrightarrow[#1]{\;#2\;}}

\begin{document}

\numberwithin{equation}{section}

\title{Gorenstein modules and Auslander categories}

\author{\|Dejun |Wu|\ \ \ \ \ \ \|Yongduo |Wang|}



\abstract
In this paper, some new characterizations on Gorenstein projective, injective and flat modules over commutative noetherian local ring are given. For instance, it is shown that an $R$-module $M$ is Gorenstein projective if and only if the Matlis dual $\Hom{M}{E(k)}$ belongs to Auslander category $\mathcal{B}(\widehat{R})$ and $\Ext{i\geq1}{M}{P}=0$ for all projective $R$-modules $P$.
\endabstract



\thanks
   The research has been supported by National Natural Science Foundation of
China (No. 11301242).
\endthanks

\section{Introduction}

Throughout this paper, $(R,\mathfrak{m},k)$ is a commutative noetherian local ring and an $R$-complex is
a complex of $R$-modules. The derived category is written
$\mathcal{D}(R)$. A complex
\begin{equation*}
  X:\quad \cdots \lra X_{i+1} \xra{\dif[i+1]{X}} X_i \xra{\dif[i]{X}}
  X_{i-1} \lra \cdots
\end{equation*}
is called \emph{acyclic} if the homology complex $\H{X}$ is the
zero-complex. We use the notations $\Cy[i]{X}$ for the kernel of
differential $\dif[i]{X}$ and $\Co[i]{X}$ for the cokernel of the
differential $\dif[i+1]{X}$. The projective, injective
and flat dimensions of $X$ are abbreviated as $\pd{X}$, $\id{X}$ and
$\fd{X}$. The full subcategories $\mathcal{P}(R)$,
$\mathcal{I}(R)$ and $\mathcal{F}(R)$ of $\mathcal{D}(R)$
consist of complexes of finite projective, injective and flat dimensions. We use the standard
notations $\RHomm{-}{-}$ and $\Ltpp{-}{-}$ for
the derived Hom and derived tensor product of complexes.

An acyclic complex $T$ of projective $R$-modules is called
\emph{totally acyclic}, if the complex $\Hom{T}{Q}$ is acyclic for
every projective $R$\nobreakdash-module $Q$. An $R$-module $M$ is
called \emph{Gorenstein projective} if there exists such a totally
acyclic complex $T$ with $\Co[0]{T} \is M$. An acyclic complex $F$ of flat $R$-modules is called \emph{totally acyclic}, if the complex $\tp{J}{F}$ is acyclic for
every injective $R$\nobreakdash-module $J$. An $R$-module $N$ is
called \emph{Gorenstein flat} if there exists such a totally
acyclic complex $F$ with $\Co[0]{F} \is N$. In \cite{Esmkhani2007J}, Esmkhani and Tousi proved that an $R$-module $M$ is Gorenstein projective if and only if $\tp{\widehat{R}}{M}\in\mathcal{A}(\widehat{R})$ and $\Ext{i\geq1}{M}{P}=0$ for all projective $R$-modules $P$; an $R$-module $M$ is Gorenstein flat if and only if $\tp{\widehat{R}}{M}\in\mathcal{A}(\widehat{R})$ and $\Tor{i\geq1}{E}{M}=0$ for all injective $R$-modules $E$, where $\widehat{R}$ is the $\mathfrak{m}$-adic completion of $(R,\mathfrak{m},k)$ and $\mathcal{A}(\widehat{R})$ is the Auslander category of $\widehat{R}$, consisting of those homologically bounded $\widehat{R}$-complexes $X$ for which $\Ltp{D}{X}$ is a homologically bounded $\widehat{R}$-complex and the canonical morphism $X\rightarrow \RHom{D}{\Ltp{D}{X}}$ is an isomorphism in $\mathcal{D}(\widehat{R})$, where $D$ is the dualizing complex of $\widehat{R}$. Note that $\widehat{R}$ is a faithfully flat $R$-module. It is well known that $E(k)$, the injective hull of the residue field, is an $\widehat{R}$-module.   Here we have the following results.\vspace{0.2cm}\\
\textbf{Theorem A.}\  \emph{An $R$-module $M$ is Gorenstein projective if and only if the Matlis dual $\Hom{M}{E(k)}\in\mathcal{B}(\widehat{R})$ and $\Ext{i\geq1}{M}{P}=0$ for all projective $R$-modules $P$.}\vspace{0.2cm}\\
\textbf{Theorem B.}\  \emph{An $R$-module $M$ is Gorenstein flat if and only if the Matlis dual $\Hom{M}{E(k)}\in\mathcal{B}(\widehat{R})$ and $\Tor{i\geq1}{I}{M}=0$ for all injective $R$-modules $I$.}

\vspace{0.2cm}A complex $U$ of injective $R$\nobreakdash-modules is called
\emph{totally acyclic} if it is acyclic, and the complex
$\Hom{J}{U}$ is acyclic for every injective $R$-module $J$.
An $R$-module $E$ is called \emph{Gorenstein injective} if there
exists a totally acyclic complex $U$ of injective $R$-modules with
$\Cy[0]{U} \is E$. In \cite{Esmkhani2007Arch}, Esmkhani and Tousi proved that an $R$-module $M$ is Gorenstein injective if and only if $\Hom{\widehat{R}}{M}\in\mathcal{B}(\widehat{R})$, $M$ is cotorsion and $\Ext{i\geq1}{E}{M}=0$ for all injective $R$-modules $E$, where $\mathcal{B}(\widehat{R})$ is the Auslander category of $\widehat{R}$, consisting of those homologically bounded $\widehat{R}$-complexes $X$ for which $\RHom{D}{X}$ is a homologically bounded $\widehat{R}$-complex and the canonical morphism $\Ltp{D}{\RHom{D}{X}}\rightarrow X$ is an isomorphism in $\mathcal{D}(\widehat{R})$, where $D$ is the dualizing complex of $\widehat{R}$. Here we also have the next result. \vspace{0.2cm}\\
\textbf{Theorem C.}\  \emph{A finitely generated $R$-module $M$ is Gorenstein injective if and only if $\tp{\widehat{R}}{M}\in\mathcal{B}(\widehat{R})$ and $\Ext{i\geq1}{E}{M}=0$ for all injective $R$-modules $E$.}

\section{Main results}

Let $(R,\mathfrak{m},k)$ be a commutative noetherian local ring and $M$ an $R$-module. We use the notation $M^v$ for the Matlis dual $\Hom{M}{E(k)}$ of $M$. There is a natural homomorphism $\varphi:M\rightarrow M^{vv}$ defined by $\varphi(x)(f)=f(x)$ for $x\in M$ and $f\in M^v$. It is well known that the canonical map $\varphi$ is an embedding; see \cite[3.4]{Enochs}. Recall that the \emph{Gorenstein injective dimension}, $\Gidd{M}$, of an $R$-module $M$ is defined by declaring that $\Gidd{M}\leq n$ if and only if $M$ has a Gorenstein injective resolution of length $n$; see \cite[Definition 2.8]{Holm2004}.

\begin{thm}\label{thm:GP}
An $R$-module $M$ is Gorenstein projective if and only if the Matlis dual $\Hom{M}{E(k)}\in\mathcal{B}(\widehat{R})$ and $\Ext{i\geq1}{M}{P}=0$ for all projective $R$-modules $P$.
\end{thm}

\begin{proof}
Let $M$ be a Gorenstein projective $R$-module. By \cite[Ascent table II(e)]{Christensen2009}, $\Hom{M}{E(k)}$ is a Gorenstein injective $\widehat{R}$-module. Hence it follows from \cite[Theorem 4.4 and Lemma 2.1]{Christensen2006} that the Matlis dual $\Hom{M}{E(k)}\in\mathcal{B}(\widehat{R})$ and $\Ext{i\geq1}{M}{P}=0$ for all projective $R$-modules $P$.

Conversely, it is enough to show that $M$ admits a co-proper right projective resolution; see \cite[Proposition 2.3]{Holm2004}. Since $\Hom{M}{E(k)}\in\mathcal{B}(\widehat{R})$, $\Gid{\Hom{M}{E(k)}}$ is finite by \cite[Theorem 4.4]{Christensen2006}. It follows from \cite[Lemma 2.18]{Christensen2006} that there is an exact sequence of $\widehat{R}$-modules $H\rightarrow \Hom{M}{E(k)}\rightarrow 0$, where $\idd{H}=\Gid{\Hom{M}{E(k)}}$. An application of $\Hom{-}{E(k)}$ yields the exact sequence
$0\rightarrow M^{vv}\rightarrow \Hom{H}{E(k)}.$
By \cite[Theorem 4.5(I)]{Avramov1991}, one has $\fdd{\Hom{H}{E(k)}}\leq \idd{H}$ and so $\fdd{\Hom{H}{E(k)}}$ is finite. Since every flat $\widehat{R}$-module is also flat as an $R$-module, the flat dimension of $\Hom{H}{E(k)}$ is finite as an $R$-module. Consequently, there exists a monomorphism $\alpha:M\rightarrow \Hom{H}{E(k)}$ with $\fd{\Hom{H}{E(k)}}$ finite. By \cite[Proposition 6.5.1]{Enochs}, there is a flat preenvelope $f:M\rightarrow F$.
Next we show that $f$ is an $\mathcal{F}(R)$-preenvelope.

Let $\psi:M\rightarrow L$ be an $R$-homomorphism such that $\fd{L}$ is finite and let $0\rightarrow K\rightarrow F'\rightarrow L\rightarrow 0$ be an exact sequence such that $\pi:F'\rightarrow L$ is a flat cover. Clearly, $K$ is of finite flat dimension and so $K$ is of finite projective dimension. By hypothesis and induction on projective dimension of $K$, one has $\Ext{i\geq1}{M}{K}=0$. Thus, one has the following exact sequence
$$0\longrightarrow\Hom{M}{K}\longrightarrow\Hom{M}{F'}\longrightarrow
\Hom{M}{L}\longrightarrow0.$$
Therefore, there exists an $R$-homomorphism $h:M\rightarrow F'$ such that $\pi h=\psi$. Since $f:M\rightarrow F$ is a flat preenvelope, there is an $R$-homomorphism $g:F\rightarrow F'$ such that $h=gf$. Thus, one has $\pi gf=\psi$ and so $f$ is an $\mathcal{F}(R)$-preenvelope. Hence there exists an $R$-homomorphism $\theta:F\rightarrow\Hom{H}{E(k)}$ such that $\theta f=\alpha$. Note that $f$ is monic for $\alpha$ is a monomorphism.

Next we show that there exists a monic $\mathcal{P}(R)$-preenvelope $M\rightarrow P$ with $P$ projective. It is easy to see $f$ is also a $\mathcal{P}(R)$-preenvelope. Notice that $f:M\rightarrow F$ is monic as $\alpha:M\rightarrow \Hom{H}{E(k)}$ is monic and $\pd{\Hom{H}{E(k)}}$ is finite. Let $0\rightarrow A\rightarrow P\rightarrow F\rightarrow0$ be an exact sequence with $P$ projective. Clearly, one has $\pd{A}<\infty$. It is easy to see that $\Ext{i\geq1}{M}{A}=0$. Therefore, there exists a monic $\mathcal{P}(R)$-preenvelope $M\rightarrow P$ with $P$ projective.

Now consider the following exact sequence $$\xymatrix{
 &0 \ar[r]  &M \ar[r]^{\beta}  &P  \ar[r] &C  \ar[r] & 0, }$$
where $\beta$ is a $\mathcal{P}(R)$-preenvelope, $P$ is a projective $R$-module and $C=\Coker{\beta}$. Let $Q$ be a projective $R$-module. Applying the functor $\Hom{-}{Q}$ to the above exact sequence, one has $\Ext{i\geq1}{C}{Q}=0$ for $\beta:M\rightarrow P$ is a $\mathcal{P}(R)$-preenvelope. It is not hard to see that $\Hom{C}{E(k)}\in\mathcal{B}(\widehat{R})$. Now proceeding in this manner, one could get the desired co-proper right projective resolution of $M$. This completes the proof.
\end{proof}

\begin{thm}
An $R$-module $M$ is Gorenstein flat if and only if the Matlis dual $\Hom{M}{E(k)}\in\mathcal{B}(\widehat{R})$ and $\Tor{i\geq1}{I}{M}=0$ for all injective $R$-modules $I$.
\end{thm}

\begin{proof}
Let $M$ be a Gorenstein flat $R$-module. By \cite[Ascent table II(d)]{Christensen2009}, the Matlis dual $\Hom{M}{E(k)}$ is a Gorenstein injective $\widehat{R}$-module. Hence it follows from \cite[Theorem 4.4 and Lemma 2.3]{Christensen2006} that $\Hom{M}{E(k)}\in\mathcal{B}(\widehat{R})$ and $\Tor{i\geq1}{I}{M}=0$ for all injective $R$-modules $I$.

Conversely, it is enough to show that $M$ admits a co-proper right flat resolution; see \cite[Theorem 3.6]{Holm2004}. By analogy with the proof of \thmref{GP}, there exists a monic $\mathcal{F}(R)$-preenvelope $f:M\rightarrow F$ with $F$ flat. Hence one has the following exact sequence
$$\xymatrix{
 &0 \ar[r]  &M \ar[r]^{f}  &F  \ar[r] &C  \ar[r] & 0, }$$
where $C=\Coker{f}$. If $F'$ is a flat $R$-module, then one has the exact sequence
$$0\rightarrow\Hom{C}{F'}\rightarrow\Hom{F}{F'}\rightarrow\Hom{M}{F'}\rightarrow0$$
for $f$ is an $\mathcal{F}(R)$-preenvelope. Since $\Hom{I}{E(k)}$ is a flat $R$-module, the next sequence
\begin{equation*}
\begin{split}
\xymatrix{
& 0 \ar[r]  & \Hom{C}{\Hom{I}{E(k)}}  }\hspace{-0.2cm}& \xymatrix{\ar[r] & \Hom{F}{\Hom{I}{E(k)}} }\\
 & \xymatrix{ \ar[r]  & \Hom{M}{\Hom{I}{E(k)}} \ar[r] & 0   }
\end{split}
\end{equation*}
is exact. By adjointness, one also has the following exact sequence
$$0\rightarrow\Hom{\tp{C}{I}}{E(k)}\rightarrow\Hom{\tp{F}{I}}{E(k)}\rightarrow
\Hom{\tp{M}{I}}{E(k)}\rightarrow0.$$
Since $E(k)$ is an injective cogenerator, one has the next exact sequence
$$0\longrightarrow\tp{M}{I}\longrightarrow\tp{F}{I}\longrightarrow\tp{C}{I}\longrightarrow0.$$
Consequently, one has $\Tor{i\geq1}{I}{C}=0$ as $\Tor{i\geq1}{I}{M}=0$. It is not hard to see that $\Hom{C}{E(k)}\in\mathcal{B}(\widehat{R})$. Now proceeding in this manner, one could get the desired co-proper right flat resolution of $M$. This completes the proof.
\end{proof}

Recall that an $R$-module $M$ is \emph{Matlis reflexive} if $M\cong M^{vv}$ under the canonical homomorphism $M\rightarrow M^{vv}$. It is well known that $\widehat{R}$ is a Matlis reflexive $R$-module.

\begin{thm}
A finitely generated $R$-module $M$ is Gorenstein injective if and only if $\tp{\widehat{R}}{M}\in\mathcal{B}(\widehat{R})$ and $\Ext{i\geq1}{E}{M}=0$ for all injective $R$-modules $E$.
\end{thm}

\begin{proof}
Let $M$ be a finitely generated Gorenstein injective $R$-module. By \cite[Theorem 3.6]{Foxby2007}, $\Gid{(\tp{\widehat{R}}{M})}$ is finite. Hence it follows from \cite[Theorem 4.4 and Lemma 2.2]{Christensen2006} that $\tp{\widehat{R}}{M}\in\mathcal{B}(\widehat{R})$ and $\Ext{i\geq1}{E}{M}=0$ for all injective $R$-modules $E$.

Conversely, it is enough to show that $M$ admits a proper left injective resolution; see \cite[Proposition 2.2]{Esmkhani2007Arch}. Since $\tp{\widehat{R}}{M}\in\mathcal{B}(\widehat{R})$, one has $\Gid{(\tp{\widehat{R}}{M})}$ is finite by \cite[Theorem 4.4]{Christensen2006}. It follows from \cite[Lemma 2.18]{Christensen2006} that there exists an exact sequence of $\widehat{R}$-modules $H\rightarrow \tp{\widehat{R}}{M}\rightarrow 0$, where $\idd{H}=\Gid{(\tp{\widehat{R}}{M})}$. Tensoring the above sequence with $\widehat{R}$ yields an exact sequence
$\tp{\widehat{R}}{H}\rightarrow\tp{\widehat{R}}{\tp{\widehat{R}}{M}}\rightarrow0$.
By \cite[Theorem 2]{Belshoff1994}, one has $\tp{\widehat{R}}{\widehat{R}}\cong \widehat{R}$ as $R$-modules for $\widehat{R}$ is a Matlis reflexive $R$-module. Consequently, one has the exact sequence $H\rightarrow M\rightarrow 0$ for $\widehat{R}$ is a faithfully flat $R$-module, where $\id{H}$ is finite for every injective $\widehat{R}$-module is injective as an $R$-module. It follows from \cite[Lemma 2.3(ii)]{Esmkhani2007Arch} that there exists an epic $\mathcal{I}(R)$-precover $E\rightarrow M$ with $E$ injective. Therefore, one has the following exact sequence
$$\xymatrix{
 &0 \ar[r]  &B \ar[r]  &E  \ar[r]^f &M  \ar[r] & 0, }$$
where $f$ is an $\mathcal{I}(R)$-precover and $B=\Ker{f}$. Next we show that $B$ satisfies the given assumptions on $M$.

Since $f$ is an $\mathcal{I}(R)$-precover, it is easy to see that $\Ext{i\geq1}{I}{B}=0$, where $I$ is an injective $R$-module. Now it remains to prove that $\tp{\widehat{R}}{B}\in\mathcal{B}(\widehat{R})$. Since $\widehat{R}$ is a flat $R$-module, one has the next exact sequence
$$0\longrightarrow\tp{\widehat{R}}{B}\longrightarrow\tp{\widehat{R}}{E}\longrightarrow
\tp{\widehat{R}}{M}\longrightarrow 0.$$
Since $E$ is isomorphic to a direct summand of $E(k)^X$ for some set $X$,  $\tp{E}{\widehat{R}}$ is isomorphic to a direct summand of $\tp{E(k)^X}{\widehat{R}}\cong \tp{E_{\widehat{R}}(\widehat{R}/\widehat{\mathfrak{m}})^X}{\widehat{R}}$. Therefore, $\tp{E}{\widehat{R}}$ is an injective $\widehat{R}$-module. It follows that $\tp{\widehat{R}}{B}\in\mathcal{B}(\widehat{R})$. Now proceeding in this manner, one could get the desired proper left injective resolution of $M$. This completes the proof.
\end{proof}

\vspace{1cm}

{\small {\em Authors' addresses}: {\em Dejun Wu and Yongduo Wang}\\ Department of
Applied Mathematics, Lanzhou University of Technology,
  Lanzhou, 730050, Gansu, China\\
{\em E-mail}: \texttt{wudj@lut.cn; ydwang@lut.cn}}

\end{document}